\documentclass{siamltex}
\usepackage{graphicx}
\usepackage{amssymb, amsmath}
\usepackage{epstopdf}
\usepackage{enumerate}
\usepackage{bm}
\usepackage{color}
\usepackage[vmargin = 1.25in, hmargin = 1.25in]{geometry}
\DeclareMathOperator*{\argmax}{arg\,max}

\newcommand{\ii}{i}
\newcommand{\jj}{j}
\newcommand{\kk}{k}
\newcommand{\nn}{n}

\newtheorem{remark}{theorem}

\title{On the Lebesgue Constant of Weighted Leja Points for Lagrange Interpolation on Unbounded Domains}

\author{%
Peter Jantsch\thanks{Department of Mathematics, University of Tennessee, Knoxville, TN 37996 ({\tt pjantsch@math.utk.edu})}%
\and C.~G.~Webster\footnotemark[1]~\thanks{Department of Computational and Applied Mathematics, Oak Ridge National Laboratory, 
Oak Ridge, TN 37831.}
\and G.~Zhang\footnotemark[2]
}

\begin{document}


\maketitle

\begin{abstract}
{This work focuses on weighted Lagrange interpolation on an unbounded domain, and analyzes the Lebesgue constant for a sequence of weighted Leja points.
The standard Leja points are a nested sequence of points defined on a compact subset of the real line, and can be extended to 
unbounded domains with the introduction of a weight function $w:\mathbb{R}\rightarrow [0,1]$. Due to a simple recursive formulation in one dimension, such abscissas provide a foundation for high-dimensional approximation methods such as sparse grid collocation, deterministic least squares, and compressed sensing.
Just as in the unweighted case of interpolation on a compact domain, we use results from potential theory to prove that the Lebesgue constant for the Leja points grows subexponentially with the number of interpolation nodes.}
\end{abstract}

\section{Introduction}\label{Intro}
The Lebesgue constant for a countable set of nodes provides a measure of how well the interpolant of a function at the given points compares to best polynomial approximation of the function. We are especially interested in how this constant grows with the number of interpolation nodes, i.e., the corresponding degree of the interpolating polynomial, in an unbounded domain.
As such, in this work we analyze the Lebesgue constant for a sequence of weighted Leja points on the real axis. Leveraging results from weighted potential theory~\cite{saff1997}, and orthogonal polynomials with exponential weights~\cite{levin2001}, we show that the Lebesgue constant for the weighted Leja points grows subexponentially with the number of interpolation nodes. 

The standard Leja sequence on $[-1,1] \subset \mathbb{R}$ is defined recursively: given a point $x_0$, for $\nn=1,2,\ldots,$ define the next Leja point as
\begin{equation}\label{leja}
	x_{\nn} = \argmax_{x\in[-1,1]} \prod_{\jj=0}^{\nn-1} \left| x - x_\jj \right| .
\end{equation}
There is still some ambiguity in this definition, since the maximum may be attained at several points. For the purposes of this work, we may choose any maximizer $x_{\nn}$ without affecting the analysis. In addition, by introducing a weight function $w:\mathbb{R} \rightarrow [0,1]$, we may also define the Leja sequence for weighted interpolation on the real line. Given a point $x_0$, for $n\geq1$ we recursively define:
\begin{equation}\label{wleja}
	x_{\nn} = \argmax_{x\in\mathbb{R}} \left( w(x) \prod_{\jj=0}^{\nn-1} |x - x_\jj| \right).
\end{equation}
As above, any maximizer is suitable, so we are not worried about the ambiguity in this definition.

The works~\cite{garcia2010b, narayan2014} show that a contracted version of the weighted Leja sequence~\eqref{wleja} is asymptotically Fekete. Specifically, this means that we first multiply the weighted Leja sequence by a contraction factor, i.e., 	
	\begin{equation}\label{contracted}
		x_{\nn,\jj} := {\nn}^{-1/\alpha} x_\jj, \quad \jj=0,\ldots,\nn,
	\end{equation}
for some appropriate real number $\alpha = \alpha(w)>1$, depending on the weight $w$. The discrete point-mass measures $\mu_{\nn}$ giving weight $1/(\nn+1)$ to each of the first $\nn+1$ \emph{contracted} Leja points, i.e.,
	\begin{equation}\label{leja_meas}
		\mu_{\nn} := \frac1{\nn+1} \sum_{\jj=0}^\nn \delta_{\{ x_{\nn,\jj} \}},
	\end{equation}
converge weak$^*$, as $n\rightarrow\infty$, to an equilibrium measure on a compact subset of $\mathbb{R}$. In other words, the Leja points asymptotically distribute similar to Fekete points, which are known to be a ``good'' set of points for interpolation (see \S\ref{ssec:Fekete}). In fact, the asymptotically Fekete property is a necessary, but not sufficient, property for a set of points to have a subexponentially growing Lebesgue constant, and motivates our study of the weighted Leja sequence for Lagrange interpolation.

The rest of the paper is organized as follows. In~\S\ref{prob}, we introduce the concept of weighted Lagrange interpolation of a function on the real line, and in Theorem~\ref{thm:main} state our main result that describes the growth of the Lebesgue constant for weighted Leja points. To prove our new theorem, we use results from potential theory, which we introduce in~\S\ref{sec:wpotential}. Specifically, we exploit the relationship between discrete potentials and polynomials with zeros at the Leja points, and the fact that the measures $\mu_{\nn}$ converge weak$^*$ to the appropriate equilibrium measure of the Fekete points. While potential theory gives us almost the whole result, we also require some explicit estimates on the spacing of the weighted Leja points, which are given in~\S\ref{sec:space}. The completion of the proof of our main theorem describing the growth of the Lebesgue constant for weighted Leja points is given in~\S\ref{sec:proof}, followed by concluding remarks.

\section{Lagrange Interpolation and Leja Points}\label{prob}
In this section we introduce the problem of weighted Lagrange interpolation of a function on the real line. We also discuss the Lebesgue constant for a set of interpolation points, and show how it relates to the best approximation error. Finally, in~\S\ref{ssec:statement} we describe our main contribution, which involves a theoretical estimate of the growth of the Lebesgue constant of the weighted Leja sequence versus of the number of interpolation points. More specifically, in Theorem~\ref{thm:main} we prove that the Lebesgue constant of the weighted Leja points grows subexponentially.

To make the setting precise, assume we are given a continuous function $f$ on $\mathbb{R}$ that we would like to interpolate. In other words, we have a set of $\nn+1$ points, $\{ x_\kk \}_{\kk=0}^{\nn} \subset \mathbb{R}$, and the values $\{ f(x_\kk) \}_{\kk=0}^{\nn}$ at each of those points. Lagrange interpolation constructs a polynomial $\mathcal{I}_\nn[f]$, of degree $\nn$, that matches $f$ at every interpolation point, i.e.,
\begin{equation*}
	\mathcal{I}_\nn[f](x_\kk) = f(x_\kk), \quad\kk=0,\dots,\nn.
\end{equation*}
The fundamental Lagrange basis functions for $\{ x_\kk \}_{\kk=0}^{\nn}$ are defined as: 
\begin{equation}
	l_{\nn,\kk} (x) = \prod_{\substack{\jj=0 \\ \jj\neq\kk}}^{\nn} \frac{ (x-x_\jj)}{(x_\kk - x_\jj)}, \qquad \kk = 0,\ldots,\nn.
\end{equation}
These functions satisfy $l_{\nn,\kk}(x_\jj) = \delta_{\jj,\kk}$ for all $\jj,\kk=0,\ldots,\nn$. The unique Lagrange interpolant of degree $\nn$ for $f$ is then given by
\begin{equation}
	\mathcal{I}_\nn[f] (x) = \sum_{\kk=0}^{\nn} f(x_\kk) l_{\nn,\kk}(x).
\end{equation}

Given an appropriate weight function $w:\mathbb{R}\rightarrow[0,1]$, to estimate the $w$-weighted approximation error for this interpolation scheme, we define $\mathbb{P}_{\nn} = \textrm{span}\{x^\jj\}_{\jj=0}^{\nn}$ to be the space of polynomials of degree at most $\nn$ over $\mathbb{R}$, and let $p_{\nn}$ be an arbitrary element of $\mathbb{P}_{\nn}$. Then the error in the norm of $L^{\infty}(\mathbb{R})$, with $\|\cdot\|_{\infty} := \|\cdot\|_{L^\infty(\mathbb{R})}$, is given by
\begin{align}
	\| w \left( f - \mathcal{I}_\nn[f] \right) \|_{\infty} &\leq \| w\left(f - p_{\nn}\right) \|_{\infty}  + \|w\, \mathcal{I}_\nn[p_{\nn}-f] \|_{\infty}  \leq \|w\left(f - p_{\nn}\right)\|_{\infty}  \left( 1 + \mathbb{L}_\nn  \right),\label{interp_p}
\end{align}
where the quantity
\begin{equation}
	\mathbb{L}_\nn := \sup_{x\in\mathbb{R}}\left(  \sum_{\kk=0}^{\nn} \frac{ w(x) | l_{\nn,\kk}(x) |}{w(x_\kk)} \right)
\end{equation}
is called the Lebesgue constant. In contrast to the case of unweighted Lagrange interpolation on a bounded domain, here the Lebesgue constant explicitly involves the weight function $w$.

In the inequality~\eqref{interp_p}, we may take the infimum over all $p_{\nn}\in\mathbb{P}_{\nn}$, to see that the Lebesgue constant relates the error in interpolation to the best approximation error by a polynomial in $\mathbb{P}_{\nn}$:
\begin{equation}
	\| w \left( f - \mathcal{I}_\nn[f] \right) \|_{\infty} \leq \left( 1 + \mathbb{L}_\nn \right) \inf_{p_{\nn}\in\mathbb{P}_{\nn}} \|w \left(f - p_{\nn}\right)\|_{\infty}.
\end{equation}
Thus, we see that the problem of constructing a stable and accurate Lagrange interpolant consists in the construction of a set of interpolation points for which $\mathbb{L}_\nn$ does not grow too quickly. 

\subsection{Our contribution}\label{ssec:statement}
In this work we prove the following result:

\begin{theorem}\label{thm:main}
Let $\alpha>1$ and assume $w:\mathbb{R}\rightarrow[0,1]$ is a weight function of the following form
\begin{equation}\label{w}
	w(x) = \exp( - Q(x) ), \quad\text{with}\quad Q(x) = |x|^{\alpha}, \quad x\in \mathbb{R}.
\end{equation}
	Then the Lebesgue constant for the weighted Leja sequence~\eqref{wleja}, defined on $\mathbb{R}$,  grows subexponentially with respect to the number of interpolation points $n$ , i.e.,
	\begin{equation*}
		\lim_{n\rightarrow\infty} \left( \mathbb{L}_\nn \right)^{\frac{1}{n}} = \lim_{n\rightarrow\infty}\left\{ \sup_{x\in\mathbb{R}} \left(\sum_{\kk=0}^{\nn} \left| \frac{ w(x) \prod_{{\jj=0,\; \jj\neq \kk}}^{\nn} (x - x_\jj) }{ w(x_\kk) \prod_{{\jj=0, \jj\neq \kk}}^{\nn} (x_\kk - x_\jj) }\right| \right) \right\}^{\frac{1}{\nn}} = 1.
	\end{equation*}
\end{theorem}

The rest of this paper is devoted to the proof of Theorem~\ref{thm:main}. Similar to the case of unweighted Leja points~\cite{taylor2008,taylor2010}, in~\S\ref{sec:wpotential}, we explore the connection between polynomials and weighted potentials, and show how classical weighted potential theory can be used to understand the asymptotic behavior (with respect to $n$) of an $n^{th}$ degree polynomial with roots at the contracted Leja points. While these techniques give us most of the result, the final part of the proof requires an explicit estimate on the spacing of the weighted Leja nodes, which is developed in~\S\ref{sec:space}. Finally, in~\S\ref{sec:proof}, we combine the spacing result and weighted potential theory to complete the proof of Theorem~\ref{thm:main}.

\section{Weighted Potential Theory}\label{sec:wpotential}

In this section, we state some necesary definitions and results from weighted potential theory, which will be the main tools we use to prove Theorem~\ref{thm:main}. For more details, we refer the interested reader to~\cite{saff1997}. 
The class of weights used in this paper, defined in~\eqref{w}, are a subset of the well-studied \emph{Freud weights}~\cite{levin2001}. From~\eqref{w}, note first that we may extend $Q$ to be a function on $\mathbb{C}$, and that $w$ has the following properties:
	\begin{enumerate}
		\item The extended weight function $w:\mathbb{C}\rightarrow[0,1]$ is continuous in $\mathbb{C}$.
		\item The set $\Sigma_0 := \{ x\in \mathbb{R} \,\big|\, w(x) > 0 \}$ has positive capacity, i.e.,
		\begin{equation*}
			\textrm{cap}(\Sigma_0) = \sup\{ \textrm{cap}(K) : K \subseteq \Sigma_0, K \textrm{ compact} \} > 0,
		\end{equation*}
		where
		\begin{equation*}
			\textrm{cap}(K) = \exp \left\{ \inf\left( \int_K \int_K {\log |x-t|} \,d\mu(x) d\mu(t): \mu \in\mathcal{M}(K) \right) \right\}.
		\end{equation*}
		\item The limit $|x| w(x) \rightarrow 0$ as $|x|\rightarrow\infty$, $x\in \mathbb{R}$. 
	\end{enumerate}
In the language of weighted potential theory, these properties imply that $w$ is \emph{admissible}.

Furthermore, we also define the Mhaskar-Rhamanov-Saff number $a_\nn = a_\nn(w)$, as the unique solution to the equation (see ~\cite[Corollary IV.1.13]{saff1997}):
\begin{equation}\label{andef}
	n = \frac1{\pi} \int_{-a_\nn}^{a_\nn}{ \frac{x\, Q'(x)}{\sqrt{a_\nn^2 - x^2}}\,dx}.
\end{equation}
 This number $a_\nn$ has a few special properties which we use in the following analysis. First, the weighted sup-norm of an $n^{th}$ degree polynomial on $\mathbb{R}$ is realized on the compact set $[-a_\nn,a_\nn]$, i.e., for all $p_\nn\in\mathbb{P}_\nn$,
\begin{equation}\label{restrictedrange}
	\|p_\nn w\|_{\infty} = \sup_{|x|\leq a_\nn} |p_\nn(x)| w(x),
\end{equation}
and $|p_\nn(x)| w(x)<\|p_\nn w\|_{\infty} $ for $|x|>a_\nn$~\cite{saff1997}. Second, from~\cite[p. 27]{levin2001}, $a_\nn\rightarrow\infty$ at approximately the rate $\nn^{1/\alpha}$, i.e.,
\begin{equation}\label{a_\nn}
	a_\nn \sim \nn^{1/\alpha}.
\end{equation}
Here, and in what follows, for two sequences $a_\nn,b_\nn$, we write $a_\nn \sim b_\nn$ if and only if there exist constants $C_1,C_2>0$, independent of $n$, such that $C_1 \leq {a_\nn}/{b_\nn} \leq C_2 $.

Let $\mathcal{M}(\mathbb{R})$ be the collection of all positive unit Borel measures $\mu$ with $\textrm{Supp}(\mu) \subseteq \mathbb{R}$. For $\mu\in\mathcal{M}(\mathbb{R})$ and $x,t\in\mathbb{R}$, define the weighted energy integral
\begin{align*}
	I_w(\mu) &= \int \int \log\left( \left| x - t \right| w(x) w(t) \right)^{-1} \,d\mu(x) d\mu(t) \\
		&= \int \int \log\frac1{|x-t|} \,d\mu(x) d\mu(t) + 2 \int Q\, d\mu.
\end{align*}
We also define the logarithmic potential by
\begin{equation}
	U^{\mu} (x) := \int \log\frac1{|x-t|} \,d\mu(t).
\end{equation}
 The goal of weighted potential theory is to find and analyze the measure $\mu\in\mathcal{M}(\mathbb{R})$ that minimizes the weighted energy integral $I_w(\mu)$. The following theorem may be found in general form in~\cite[Theorem I.1.3]{saff1997}, and is presented here for the specific case~\eqref{w} of a continuous, admissible weight $w$ on $\mathbb{R}$. 
\begin{theorem}\label{thm:potential}
Let $w$ be a continuous, admissible weight function on $\mathbb{R}\subset\mathbb{C}$, and define
\begin{equation}\label{Vw}
	V_w := \inf\left\{ I_w(\mu) \,\big|\, \mu\in\mathcal{M}(\mathbb{R})\right\}.
\end{equation}
Then we have the following properties:
\begin{itemize}
	\item The quantity $V_w$ is finite.
	\item There exists a unique measure $\mu_w \in \mathcal{M}(\mathbb{R})$ such that
		\begin{equation*}
			I_w(\mu_w) = V_w,
		\end{equation*}
		and the equilibrium measure $\mu_w$ has finite logarithmic energy, i.e.,
		\begin{equation*}
			-\infty < \int\int \log\frac1{|x-t|} \,d\mu_w(t)d\mu_w(x)=\int U^{\mu_w}(x) \,d\mu_w(x) < \infty.
		\end{equation*}
	\item Let $F_w$ be the modified Robin constant for $w$, given by
	\begin{equation}\label{robin}
		F_w := V_w - \int Q\, d\mu_w.
	\end{equation}		
	The logarithmic potential $U^{\mu_w}$ is continuous for $z\in\mathbb{C}$ and, moreover, for every $x\in \textrm{Supp}(\mu_w)\subset\mathbb{R}$,
		\begin{equation} \label{UQF}
			U^{\mu_w}(x) + Q(x) = F_w.
		\end{equation}

\end{itemize}
\end{theorem}
\begin{proof}
	The first two statements are quoted directly from, and proved in,~\cite[Theorem I.1.3]{saff1997}. To prove the third statement, we note that $\mathbb{C}\setminus\mathbb{R}$ has exactly two connected components, namely $\{ \text{Im}(z)>0 \}$ and $\{ \text{Im}(z)<0 \}$, and that of course every point in $\textrm{Supp}(\mu_w) \subset \{ \text{Im}(z)=0 \}$ is a boundary point for both of these sets. Thus, by~\cite[Theorem I.5.1]{saff1997}, $U^{\mu_w}$ is continuous on $\textrm{Supp}(\mu_w)$. Hence, from~\cite[Theorem I.4.4]{saff1997},  $U^{\mu_w}$ is continuous on all of $\mathbb{C}$, and~\eqref{UQF} holds for every $x\in \textrm{Supp}(\mu_w) \subset\mathbb{R}$. 
\end{proof}

\subsection{Weighted Fekete Points}\label{ssec:Fekete}
In this section we describe the connection between Leja points and the weighted equilibrium measure $\mu_w$. For $\nn\geq0$, let $\mathcal{T}_\nn$ denote a general set of points in $\mathbb{R}$ with cardinality $|\mathcal{T}_\nn|=\nn+1$, and let $w$ be an admissible weight on $\mathbb{R}$. We say a set of $\nn+1$ points is (weighted-)Fekete if it maximizes the quantity:
	\begin{equation}\label{Fekete}
		\mathcal{F}_\nn = \argmax_{|\mathcal{T}_\nn|=\nn+1}\Bigg( \prod_{{t,s \in \mathcal{T}_\nn \atop t\neq s}} |t-s|w(t)w(s) \Bigg)^{\frac2{(\nn+1)(\nn+2)}}.
	\end{equation}
It is known that the Lebesgue constant for a set of Fekete points $\mathcal{F}_\nn$ satisfies
	\begin{equation*}
		\mathbb{L}(\mathcal{F}_\nn) := \sup_{x\in \mathbb{R}} \left( \sum_{s\in\mathcal{F}_\nn} \bigg| \frac{ w(x) \prod_{t\neq s} (x - t) }{ w(s) \prod_{t \neq s} (s - t) }\bigg|\right) \leq \nn + 1.
	\end{equation*}
Furthermore, we also know that for a sequence of Fekete point sets, $\{\mathcal{F}_\nn\}_{n\geq1}$,
	\begin{equation*}
		\lim_{n\rightarrow\infty}\Bigg( \prod_{{t,s \in \mathcal{F}_\nn\atop t\neq s}} |t-s|w(t)w(s) \Bigg)^{\frac2{(\nn+1)(\nn+2)}}= \exp(-V_w),
	\end{equation*}
where $V_w$, as  defined in~\eqref{Vw}, is the weighted logarithmic capacity for $\mathbb{R}$ with respect to $w$. For interpolation schemes, we are also interested in arrays of points with similar asymptotic properties to Fekete points in the limit as $n\rightarrow\infty$, since this is a necessary condition for a sequence of points to have a well-behaved Lebesgue constant. Thus, we make the following definition:
\begin{definition}
	A sequence of point sets $\{\mathcal{T}_\nn\}_{n\geq1}$, with $|\mathcal{T}_\nn| = n, \, n\geq1$, is called 
	\\asymptotically (weighted) Fekete if
	\begin{equation*}
		\lim_{n\rightarrow\infty} \Bigg( \prod_{{t,s \in \mathcal{T}_\nn\atop t\neq s}} |t-s|w(t)w(s) \Bigg)^{\frac2{(\nn+1)(\nn+2)}} = \exp(-V_w).
	\end{equation*}
\end{definition}
Note that a sequence of interpolation points may be \emph{asymptotically} Fekete but not Fekete, i.e., without satisfying~\eqref{Fekete} for any $n\in\mathbb{N}$. The following lemma, first proved in~\cite{garcia2010b} in a more general setting than the one considered here, and later in
\cite{narayan2014}, indicates that the contracted Leja sequence distributes asymptotically like the Fekete points.
\begin{lemma}\label{fekete}
	The contracted Leja sequence, defined by~\eqref{wleja} and~\eqref{contracted} is asymptotically Fekete. 
\end{lemma}

Next we define the discrete point-mass measure associated with the points $\mathcal{T}_\nn$ as
\begin{equation*}
	\nu_{\mathcal{T}_\nn} = \frac1{\nn+1} \sum_{t\in\mathcal{T}_\nn} \delta_{\{t\}},
\end{equation*}
where $\delta_{\{t\}}$ is the standard Dirac delta function for the point $t \in\mathcal{T}_\nn$. If a sequence of measures $\{\nu_{\mathcal{T}_\nn}\}_{n\geq0}$ corresponds to an asymptotically Fekete sequence of interpolation nodes, the next lemma tells us that they converge to a particular measure; see~\cite[Theorem 2.3]{garcia2010b}, and~\cite[Theorem 3.1]{narayan2014}.
\begin{lemma}
	Let $\mu_w$ be the equilibrium measure for $\mathbb{R}$ with respect to $w$ (see Theorem~\ref{thm:potential}), and let $\{\mathcal{T}_\nn\}_{n\geq 0}$ be an asymptotically Fekete sequence of point sets with corresponding discrete measures $\{\nu_{\mathcal{T}_\nn}\}_{n\geq0}$. Then we have
	\begin{equation*}
		\lim_{n\rightarrow\infty} \nu_{\mathcal{T}_\nn} = \lim_{n\rightarrow\infty} \frac1{\nn+1} \sum_{t\in\mathcal{T}_\nn} \delta_{\{t\}} = \mu_w,
	\end{equation*}
	where equality is understood in the weak$^*$ sense. In particular, for the measures $\mu_\nn$, defined by~\eqref{leja_meas}, corresponding to the contracted Leja sequence,
	\begin{equation*}
		\lim_{\nn\rightarrow\infty} \mu_\nn = \mu_w.
	\end{equation*}
\end{lemma}

\subsection{Potentials and Polynomials}
Taken together, the previous two lemmas tell us that the discrete point-mass measures associated with the contracted Leja sequence converge weak$^*$ to the weighted equilibrium measure for $\mathbb{R}$ corresponding to the weight $w$ given in~\eqref{w}. This fact enables us to make a key connection between potential theory and Leja points, and provides the basis for the proof of Theorem~\ref{thm:main}.

With $\{x_{\nn,\jj}\}_{\jj=0}^\nn$ as in~\eqref{contracted}, define $P_{\nn,\kk}$ to be the polynomial with roots at each of the $\nn$ contracted Leja points $x_{\nn,\jj}, \jj=0,\ldots,\kk-1,\kk+1,\ldots, \nn$, i.e.,
	\begin{equation*}
		P_{n,k}(x) = \prod_{{\jj=0\atop \jj\neq \kk}}^{\nn} (x - x_{\nn,\jj}),
	\end{equation*}
and let $\mu_{n,k}$ be the measure which assigns mass $\frac1{\nn}$ to each of the roots of $P_{\nn,\kk}$, i.e.,
\begin{equation}
	\mu_{\nn,\kk} = \frac1{\nn} \sum_{{\jj=0\atop \jj\neq \kk}}^\nn \delta_{\{x_{\nn,\jj}\}}.
\end{equation}
Then, taking the logarithm of $| P_{\nn,\kk}^{1/n} w |$, we convert the polynomial into a discrete logarithmic potential with respect to the measure $\mu_{\nn,\kk}$, i.e.,
\begin{align*}
	\log\left| P_{n,k}(x)w(x)^n \right|^{\frac{1}{n}} &= \frac1{\nn}\sum_{{\jj=0\atop \jj\neq \kk}}^n \log|x - x_{\nn,\jj}| - Q(x) = -U^{\mu_{n,k}}(x) - Q(x). 
\end{align*}
By Lemma~\ref{fekete}, the weighted Leja sequence is asymptotically Fekete, and therefore we have $\mu_{n,k} \rightarrow \mu_w$ in the weak$^*$ sense. This connections allows us to exploit potential theory to understand the asymptotic behavior of weighted polynomials. In particular, by considering polynomials with roots at the contracted Leja points~\eqref{contracted}, we explicitly explore this asymptotic behavior in the following two lemmas, which will be an essential part of the proof of Theorem~\ref{thm:main}.
\begin{lemma}\label{lem:numer}
	Given $\varepsilon>0$, there exists an $N\in\mathbb{N}$ such that, for $n>N$ and $0\leq k \leq n$,
	\begin{equation*}
		\left| \| P_{n,k} w^n \|_{\infty}^{\frac{1}{n}} - \exp{( -F_w )}\right| <\varepsilon.
	\end{equation*}
\end{lemma}
\begin{proof}
See Appendix~\ref{ssec:numer}.
\end{proof}
\begin{lemma} \label{lem:P1}
	For all $\varepsilon>0$, there exist $\delta>0$ and $N\in\mathbb{N}$, such that for $n>N$, and $0\leq \kk\leq\nn$,
	\begin{equation*}
		\left| \Bigg( w(x_{\nn,\kk})^n \prod_{|x_{\nn,\kk} - x_{\nn,\jj}| \geq \delta} |x_{\nn,\kk} - x_{\nn,\jj}|  \Bigg)^{\frac{1}{n}} - \exp(-F_w)\right| < \varepsilon.
	\end{equation*}
\end{lemma}
\begin{proof}
See Appendix~\ref{ssec:P1}.
\end{proof}

\section{Spacing of the weighted Leja points}\label{sec:space}

The goal of this section is to state and prove a result regarding the spacing of the contracted Leja sequence. This will be crucial to the final step in the proof of Theorem~\ref{thm:main}.

\begin{theorem}\label{thm:space}
Let $w$ and $\alpha>1$ be as in~\eqref{w}, and let $n\in\mathbb{N}$, with $0\leq i,\jj\leq n$. Then, for some constant $C>0$, independent of $n$, the contracted Leja sequence~\eqref{contracted} satisfies the spacing property
\begin{equation}
	C|x_{\nn,\ii} - x_{\nn,\jj}| \geq \nn^{-1}.
\end{equation}

\end{theorem}

To prove Theorem~\ref{thm:space}, the main spacing result for the contracted Leja sequence, we use a weighted version of the classical Markov-Bernstein inequalities, which relate norms of polynomials to norms of their derivatives. First, for $a_\nn$ and $Q$ as defined in~\eqref{andef} and~\eqref{w}, respectively, define the function
\begin{equation}\label{varphi}
	\varphi_\nn(t) = \frac{ |t-a_{2n}| |t + a_{2n}| } { n \sqrt{(|t+a_{\nn}| - a_\nn\zeta_\nn)(|t-a_{\nn}| + a_\nn\zeta_\nn)}},
\end{equation}
where
\begin{equation*}
	\zeta_\nn = \left( \alpha \nn\right)^{-2/3}.
\end{equation*}
\begin{remark} The function $\varphi_\nn$ plays the same role as the function 
\begin{equation*}
	\phi_\nn(t) = \frac1{ n \sqrt{1-t^2}},
\end{equation*}
for the Markov-Bernstein inequalities for unweighted polynomials on $[-1,1]$. 
\end{remark}

\begin{proof}[Proof of Theorem~\ref{thm:space}]\,
Let $\varphi$ be as in~\eqref{varphi}. The main fact we need for this proof is a Bernstein-type inequality for weighted polynomials, which can be found, for instance, in~\cite[Theorem 10.1]{levin2001}: for any polynomial $p_\nn$ of degree $n\geq1$, there exists some $C$, independent of $p_\nn$ and $n$, such that
\begin{equation}\label{MB1}
	\left| (p_\nn(t)w(t))' \right| \leq \frac{C}{\varphi_\nn(t)} \| p_\nn w \|_{\infty}, \quad t\in\mathbb{R}.
\end{equation}
 From~\cite[Theorem 5.4(b)]{levin2001}, we estimate that
\begin{equation*}
	\sup_{t\in[-a_\nn,a_\nn]} \left|\frac1{\varphi_{\nn}(t)}\right| \sim \sqrt{\alpha}  \frac{ n }{a_\nn}.
\end{equation*}
Hence, for any polynomial $p_{\nn}$ of degree $\nn$, and $t\in\mathbb{R}$,
\begin{equation}\label{bernstein}
	\left| (p_{\nn}(t)w(t))' \right| \leq C \frac{n}{a_\nn}  \| p_{\nn}w \|_{\infty}.
\end{equation}
In particular, this holds for the polynomial $P_{\nn}$ defined by
\begin{equation}\label{eq:Pn}
	P_{\nn}(t) := \prod_{\jj=0}^{n-1}(t-x_\jj).
\end{equation}
Given $0\leq\jj < n$, by the mean value theorem, there exists a point $t$ between $x_\jj$ and $x_\nn$ such that 
\begin{align*}
	\frac{|P_{\nn}(x_\jj)w(x_\jj) - P_{\nn}(x_\nn)w(x_\nn)|}{|x_\nn - x_\jj|} &= \left| (P_{\nn}(t) w(t))' \right|.
\end{align*}
Notice that for $0\leq \jj <\nn$, $P_{\nn}(x_\jj) = 0$ by definition. Then from~\eqref{bernstein},
\begin{equation*}
	\frac{|P_{\nn}(x_\nn)w(x_\nn)|}{|x_\nn - x_\jj|} \leq \frac{C \nn}{a_\nn}  \left| P_{\nn}(x_\nn) w(x_\nn) \right|,
\end{equation*}
which implies
\begin{equation*}
	C|x_\nn - x_\jj| \geq \frac{a_\nn}{ \nn}.
\end{equation*}
Using the fact $a_\nn \sim \nn^{1/\alpha}$ from~\eqref{a_\nn}, we get 
	\begin{align}\label{space2}
		C| x_{\nn} - x_\jj | &\geq \frac{a_{\nn} }{\nn} \sim \nn^{1/\alpha - 1}.
	\end{align}
Let $n\geq1$, and $\jj < n$, such that $x_{\nn,\jj},x_{\nn,\nn}\geq0$. Then using~\eqref{space2}, along with~\eqref{contracted}, we calculate
	\begin{align*}
		C|x_{n,n} - x_{\nn,\jj}| &= C \nn^{-1/\alpha} | x_{\nn} - x_\jj | \geq \nn^{-1/\alpha} \nn^{1/\alpha - 1} = \nn^{-1}.
	\end{align*}
Now let $i,\jj\leq n$, and assume without loss of generality that $i<\jj$.  The above calculation shows that 
	\begin{align*}
		2 C | x_{\nn,\ii} - x_{\nn,\jj} | \geq \jj^{- 1} \geq \nn^{- 1}.
	\end{align*}
	which, up to constants independent of $n$, is the desired result.
\end{proof}

\section{Proof of Theorem~\ref{thm:main}}\label{sec:proof}
In this section, we prove our main theorem concerning the growth of the Lebesgue constant of the weighted Leja sequence. Similar to the proof in the unweighted case given in~\cite{taylor2008,taylor2010}, we separate the proof of the theorem into several smaller components. 

To begin, we first show that the Lebesgue constant of the weighted Leja sequence on the real line is equal to a weighted Lebesgue constant of the contracted Leja sequence~\eqref{contracted} on a fixed compact set. To do this, we first use the fact from~\eqref{restrictedrange} that supremum a $w$-weighted, $\nn^{th}$ degree polynomial is realized in the compact set $[-a_\nn,a_\nn]$. Then, we exploit the specific form~\eqref{w} of our weight function to show that
\begin{equation}\label{Qhomo}
	Q(\nn^{1/\alpha} x) = n Q(x),
\end{equation}
which in turn implies that
\begin{equation}\label{whomo}
	w(x) = w(\nn^{-1/\alpha} x)^{\nn}.
\end{equation}
Finally, let $0<c<\infty$ be the smallest constant such that
\begin{equation}\label{c}
	\sup_{\nn} (\nn^{-1/\alpha} a_\nn) \leq c.
\end{equation}
Note that $c<\infty$ by~\eqref{a_\nn}. Now defining $K := [-c,c]$, this means that
\begin{equation}\label{compactset}
	y \in [-a_\nn,a_\nn] \implies x:= \nn^{-1/\alpha}y \in K.
\end{equation}
Furthermore, for any $\nn=1,2,\dots$, let $q_\nn \in \mathbb{P}_\nn$. Define $\widetilde{q}_n(x) \in \mathbb{P}_\nn$ to be the unique polynomial such that
\begin{equation*}
	q_\nn(x) = \nn^{-\nn/\alpha} \widetilde{q}_\nn (\nn^{1/\alpha} x )
\end{equation*}
Then we calculate
\begin{align*}
	\sup_{x\in\mathbb{R}} w(x)^\nn \left| q_n(x)\right| & = \nn^{-\nn/\alpha} \sup_{x\in\mathbb{R}} w(\nn^{1/\alpha} x)  \left| \widetilde{q}_\nn (\nn^{1/\alpha} x ) \right|\\
		& = \nn^{-\nn/\alpha}  \sup_{y\in\mathbb{R}} w(y)  \left| \widetilde{q}_\nn (y) \right| \\
		& =  \nn^{-\nn/\alpha} \sup_{y\in[-a_n,a_n]} w(y)  \left| \widetilde{q}_\nn (y) \right| \\
		& = \sup_{y\in[-a_n,a_n]} w(\nn^{-1/\alpha} y)^\nn \left|  q_\nn (\nn^{-1/\alpha} y) \right| \\
		& \leq \sup_{x\in K} w(x)^\nn  \left| q_\nn (x) \right| \\
		& \leq \sup_{x\in \mathbb{R}} w(x)^\nn  \left| q_\nn (x) \right|.
\end{align*}
From this string of inequalities we have that for any $\nn\geq 1$, and $q_\nn \in \mathbb{P}_\nn$,
\begin{equation*}
	\sup_{x\in\mathbb{R}} w(x)^\nn \left| q_n(x)\right| = \sup_{x\in K} w(x)^\nn \left| q_n(x)\right|.
\end{equation*}
Then using~\cite[Corollary III.2.6]{saff1997}, we know that $\textrm{supp}(\mu_w) =: K_w \subseteq K$, and
\begin{equation}\label{meas_supp}
	\sup_{x\in K} w(x)^\nn \left| q_n(x)\right| = \sup_{x\in K_w} w(x)^\nn \left| q_n(x)\right|.
\end{equation}

Now from the definition~\eqref{restrictedrange}, along with~\eqref{Qhomo}--\eqref{meas_supp}, we calculate
\begin{align*}
	\mathbb{L}_\nn &= \sup_{x\in\mathbb{R}} \Bigg\{ \sum_{\kk=0}^\nn \Bigg| \frac{w(x)}{w(x_{\kk})} \Bigg(\prod_{{\jj=0\atop \jj\neq \kk}}^\nn \frac{x - x_\jj}{x_\kk - x_\jj}\Bigg) \Bigg| \Bigg\}\\
		&= \sup_{x\in[-a_\nn,a_\nn]} \Bigg\{ \sum_{\kk=0}^n \Bigg| \frac{w(x)}{w(x_{\kk})} \Bigg(\prod_{{\jj=0\atop \jj\neq \kk}}^n \frac{x - x_\jj}{x_\kk - x_\jj}\Bigg) \Bigg| \Bigg\}\\
		&= \sup_{x\in[-a_\nn,a_\nn]}\Bigg\{ \sum_{\kk=0}^n \Bigg| \frac{w(\nn^{-1/\alpha}x)^n}{w(\nn^{-1/\alpha}x_{\kk})^n} \Bigg(\prod_{{\jj=0\atop \jj\neq \kk}}^n \frac{\nn^{-1/\alpha} (x - x_\jj)}{\nn^{-1/\alpha}(x_\kk - x_\jj)}\Bigg) \Bigg| \Bigg\}\\
		&\leq \sup_{y\in K} \Bigg\{ \sum_{\kk=0}^n \Bigg| \frac{w(y)^n}{w(x_{\nn,\kk})^n} \Bigg(\prod_{{\jj=0\atop \jj\neq \kk}}^n \frac{y - x_{\nn,\jj}}{x_{\nn,\kk} - x_{\nn,\jj}}\Bigg) \Bigg| \Bigg\}\\
		& \leq n\Bigg\{ \max_{\kk=0,\ldots,n} \Bigg( \frac{ \sup_{y\in K_w} | w(y)^n \prod_{{\jj=0,\, \jj\neq \kk}}^n (y - x_{\nn,\jj}) |}{ w(x_{\nn,\kk})^n\prod_{{\jj=0,\, \jj\neq \kk}}^n |x_{\nn,\kk} - x_{\nn,\jj}|}\Bigg) \Bigg\}.
\end{align*}
Thus, to show that this Lebesgue constant grows at a subexponential rate, the above calculation indicates that we only need to show that
\begin{equation} \label{LC}
	\lim_{n\rightarrow\infty} \left\{ n\left( \max_{\kk=0,\ldots,n} \frac{ \sup_{y\in K_w} | w(y)^n \prod_{{\jj=0,\,\jj\neq \kk}}^n (y - x_{\nn,\jj}) |}{ w(x_{\nn,\kk})^n\prod_{{\jj=0,\, \jj\neq \kk}}^n |x_{\nn,\kk} - x_{\nn,\jj}|} \right) \right\}^{\frac{1}{n}} = 1.
\end{equation}
Of course, $\nn^{1/n} \rightarrow 1$ as $n \rightarrow \infty$, so to prove \eqref{LC}, it is enough to show that, uniformly in $k$, the numerator and denominator both converge to $\exp(-F_w)$, i.e., 
\begin{equation}\label{numer}
	\lim_{n\rightarrow\infty} \sup_{y\in K_w} \Bigg( | w(y)^n \prod_{{\jj=0\atop \jj\neq \kk}}^n (y - x_{\nn,\jj}) | \Bigg)^{\frac{1}{n}} = \exp(-F_w),
\end{equation}
and 
\begin{equation}\label{denom}
	\lim_{n\rightarrow\infty} \Bigg( w(x_{\nn,\kk})^n\prod_{{\jj=0\atop \jj\neq \kk}}^n \left|x_{\nn,\kk} - x_{\nn,\jj}\right|  \Bigg)^{\frac{1}{n}} = \exp(-F_w),
\end{equation}
with both limits independent of $\kk=0,\ldots, \nn$.
Recall that $F_w$ was defined explicitly in~\eqref{robin}, and is called the Robin constant with respect to the weight $w$.

Let $\delta>0$, and $\kk=0,\ldots,\nn$. To prove \eqref{denom}, we split the product into two parts:
\begin{align*}
	\prod_{{\jj=0\atop \jj\neq \kk}}^n |x_{\nn,\kk} &- x_{\nn,\jj}| w(x_{\nn,\kk})  = \underbrace{\Bigg( w(x_{\nn,\kk})^n \prod_{|x_{\nn,\kk} - x_{\nn,\jj}| \geq \delta} |x_{\nn,\kk} - x_{\nn,\jj}| \Bigg)}_{A_1(\nn,\kk, \delta)} \underbrace{\Bigg( \prod_{|x_{\nn,\kk} - x_{\nn,\jj}| < \delta} |x_{\nn,\kk} - x_{\nn,\jj}| \Bigg)}_{A_2(\kk,n,\delta)}.
\end{align*}
Then we seek to show that as $n\rightarrow\infty$ and $\delta\rightarrow0$,
\begin{equation}\label{P1}
	A_1(\nn,\kk, \delta)^{1/n} \rightarrow \exp(-F_w),
\end{equation}
and
\begin{equation}\label{P2}
	A_2(\nn,\kk, \delta)^{1/n} \rightarrow 1,
\end{equation}
and that convergence of the limits is independent of $\kk = 0, \ldots, \nn$.

We have reduced the proof to essentially a problem in weighted potential theory. The convergence of the limits~\eqref{numer} and~\eqref{P1} follow directly from Lemmas~\ref{lem:numer} and~\ref{lem:P1}, respectively, which are proven in the appendix. Thus, we have left to show statement~\eqref{P2}, which requires a more direct approach. We explicitly use the spacing of the contracted Leja sequence from Theorem~\ref{thm:space}, and find that the remainder of the estimate involving $A_2(\nn, \kk,\delta)$ follows from this spacing lemma.

By assuming $\delta<1$, it is clear that the product $A_2(\nn,\kk,\delta)$ is always less than one. Therefore, the following theorem will complete the proof of Theorem~\ref{thm:main}.
\begin{lemma} Given $\varepsilon>0$, there exists $\delta>0$, $N\in\mathbb{N}$ such that for $n>N$, and $0\leq \kk\leq \nn$,
	\begin{equation*}
		\Bigg( \prod_{|x_{\nn,\kk} - x_{\nn,\jj}| < \delta} |x_{\nn,\kk} - x_{\nn,\jj}|  \Bigg)^{\frac{1}{n}} > 1 - \varepsilon.
	\end{equation*}
\end{lemma}
\begin{proof}
We first split the product into two components:
	\begin{align*}
		 \prod_{|x_{\nn,\kk} - x_{\nn,\jj}| < \delta} |x_{\nn,\kk} - x_{\nn,\jj}| = \prod_{x_{\nn,\jj}\in X_1(\kk,\delta)} |x_{\nn,\kk} - x_{\nn,\jj}| \times \prod_{x_{\nn,\jj}\in X_2(\kk,\delta)} |x_{\nn,\kk} - x_{\nn,\jj}|.%
	\end{align*}
	where
	\begin{align*}
		X_1(\kk,\delta) & := \left\{ x_{\nn,\jj} ~\Big|~ \jj\leq n,\, x_{\nn,\kk}-\delta < x_{\nn,\jj} \leq x_{\nn,\kk} \right\}, \\
		X_2(\kk,\delta) & := \left\{ x_{\nn,\jj} ~\Big|~ \jj\leq n,\,x_{\nn,\kk} \leq x_{\nn,\jj} <  \,x_{\nn,\kk} + \delta \right\}.
	\end{align*}
	At least one of these sets may be empty, and in that case we simply set the corresponding product equal to one. Now, let $m_1, m_2$ be the cardinality of the sets $X_1(\kk,\delta)$ and $X_2(\kk,\delta)$, resp., and label these points in the following way
	\begin{equation*}
		x_{\nn,\kk} - \delta \leq x_{\nn,i_{m_1}} \leq \ldots \leq x_{\nn, i_1}  < x_{\nn,\kk} < x_{\nn,j_1} < \ldots < x_{\nn,j_{m_2}} < x_{\nn,\kk} + \delta.
	\end{equation*}
	Then from Theorem~\ref{thm:space}, we can show that for any $1\leq s \leq m_1$,
	\begin{align}\label{X1est}
		 | x_{\nn,\kk} - x_{\nn, i_{s}} | = |x_{\nn,\kk} - x_{\nn,i_1}| + \ldots + |x_{\nn,i_{s-1}} - x_{\nn,i_s}| \geq \frac{s}{C n}.
	\end{align}
	Similarly, for $1\leq t \leq m_2$,
	\begin{align}\label{X2est}
		 | x_{\nn,\kk} - x_{\nn, j_{t}} | = |x_{\nn,\kk} - x_{\nn,j_1}| + \ldots + |x_{\nn,j_{t-1}} - x_{\nn,j_t}| \geq \frac{t}{C n}.
	\end{align}
	
	Now, using~\eqref{X1est} and Sterling's approximation, we see that
	\begin{align}
		\Bigg( \prod_{x_{\nn,\jj}\in X_1(\kk,\delta)} |x_{\nn,\kk} - x_{\nn,\jj}|  \Bigg)^{\frac{1}{n}} &= \left( \prod_{s=1}^{m_1} |x_{\nn,\kk} - x_{\nn,i_s}|  \right)^{\frac{1}{n}} \\
		& \geq \left( \prod_{s=1}^{m_1} \frac{s}{C \nn}   \right)^{\frac{1}{n}} \\
		& = \Bigg( \frac{m_1!^{\frac1{m_1}}}{C\nn} \Bigg)^{\frac{m_1}{n}} \\
		& \geq \left(  \frac{m_1}{C\nn} \right)^{\frac{m_1}{n}}. \label{m1}
	\end{align}
	Similarly, we can show that
	\begin{equation}
		\Bigg( \prod_{x_{\nn,\jj}\in X_2(\kk,\delta)} |x_{\nn,\kk} - x_{\nn,\jj}|  \Bigg)^{1/n}  \geq \left(  \frac{m_2}{C\nn} \right)^{m_2/n}. \label{m2}
	\end{equation}
	
	As $\tau\rightarrow0^+$, the function $(\frac{\tau}{C})^{2\tau} \rightarrow 1$. Thus, we let $\tau<\min\{C,\frac1{e}\}$ be small enough so that		
	\begin{align*}
		1-\varepsilon < \left(  \frac{\tau}{C} \right)^{2\tau} < 1.
	\end{align*}
	Let $m$ be the number of Leja points within the interval $\{t\in\mathbb{R}: |x_{\nn,\kk} - t| < \delta\}$. According to~\cite[Theorem 2.4.5]{taylor2008}, for our chosen $\tau>0$, we can choose $N\in\mathbb{N},$ and $\delta_0>0$ such that if $n>N$, and $\delta<\delta_0$,
	\begin{equation*}
		 \max\left\{ \frac{m_1}{\nn} , \frac{m_2}{\nn} \right\}\leq \frac{m}{\nn} = \int_{|t-x_{\nn,\kk}|<\delta} d\mu_{\nn,\kk}(t) < \tau.
	\end{equation*}
	We know $f(x) = x^x$ is a decreasing function on $(0,\frac1e)$, and hence from~\eqref{m1} and~\eqref{m2}, this implies that 
	\begin{align*}
		\Bigg( \prod_{x_{\nn,\jj}\in X_1(\kk,\delta)} |x_{\nn,\kk} - x_{\nn,\jj}|  \Bigg)^{1/n} & \Bigg( \prod_{x_{\nn,\jj}\in X_2(\kk,\delta)} |x_{\nn,\kk} - x_{\nn,\jj}|  \Bigg)^{1/n} \\
		 & \geq \left(  \frac{m_1}{Cn} \right)^{m_1/n} \left(  \frac{m_2}{Cn} \right)^{m_2/n} \\
		&\geq \left( \frac{\tau}{C} \right)^{2\tau}> 1-\varepsilon,
	\end{align*}
	which is the desired result for $X_1(\kk,\delta)$ and $X_2(\kk,\delta)$.		
This completes the proof. 
\end{proof}

\section{Conclusion}\label{conclusion}
In this work, we considered the properties of Leja points for weighted Lagrange interpolation on an unbounded domain. Due to their nested structure, simple recursive formulation, and generally stable behavior, Leja points show promise for high-dimensional interpolation methods. Our contribution to this area was to prove that the Lebesgue constant for the weighted Leja sequence grows subexponentially with respect to the number of interpolation nodes. Furthermore, we proved a theorem regarding the separation of the weighted Leja points.

Of course, a subexponential rate encompasses a wide range of growth, potentially much bigger than the optimal Lebesgue constant $\mathcal{O} (\log n)$. On the other hand, our experience with Leja points indicates that the Lebesgue constant grows linearly, i.e., $\mathcal{O}(n)$, with respect to the number of nodes. Our proof in this paper relies on potential theory, which gives only asymptotic estimates of growth. We expect that a more explicit estimate of the Lebesgue constant would require different techniques, and this is the subject of future work.


\bibliographystyle{siam}
\bibliography{WLeja}

\clearpage
\appendix

\renewcommand\thesection{\Alph{section}}
\setcounter{section}{1}
\section*{Appendix}\label{appendix}
\subsection{Proof of Lemma~\ref{lem:numer}}\label{ssec:numer}

\begin{proof}
First,~\cite[Theorem I.3.6]{saff1997} implies that for all $n$ and $0\leq k\leq n$,
$
	\| P_{n,k} w^n \|_{\infty} \geq \exp(-nF_w),
$
which yields that $\liminf_{n\rightarrow\infty} \| P_{n,k} w^n \|_{\infty}^{1/n} \geq \exp(-F_w)$ is independent of $k$. 
In the remainder of the proof, we seek to show that 
\[
\limsup_{n\rightarrow\infty} \| P_{n,k} w^n \|_{\infty}^{1/n} \leq \exp(-F_w).
\] 
Now for a given $\varepsilon>0$, we seek to show that there exists an $N$ such that for $n>N$ and $0\leq k\leq n$, 
\begin{equation*}
	\sup_{x\in\mathbb{R}} \left\{ \frac1{n} \log |P_{n,k}(x)| - Q(x)\right\} \leq -F_w + \epsilon.
\end{equation*}
Define $K_w := \textrm{Supp}(\mu_w)$. Because of~\eqref{meas_supp}, we know that for our weight function
\begin{equation*}
	\sup_{x\in\mathbb{R}} \left\{ \frac1{n} \log |P_{n,k}(x)| - Q(x)\right\} = \sup_{x\in K_w} \left\{ \frac1{n} \log |P_{n,k}(x)| - Q(x)\right\},
\end{equation*}
and from~\eqref{UQF},  we have the relation
\begin{equation*}
	U^{\mu_w}(x) +Q(x) = F_w, \quad \forall x\in K_w \subset\mathbb{R}.
\end{equation*}  
Hence we can write
\begin{align} \label{WTS}
	\sup_{x\in K_w} \left\{ \frac1{n} \log |P_{n,k}(x)| - Q(x)\right\} = -F_w + \sup_{x\in K_w} \left\{ -U^{\mu_{n,k}}(x) + U^{\mu_w}(x)\right\}.
\end{align}
Let $\delta>0$, to be chosen later. We rewrite the arguments of the supremum on the right-hand side of~\eqref{WTS} as integrals and divide them each into two parts: 
\begin{align*}
	-U^{\mu_{n,k}}(x) + U^{\mu_w}(x) =& \int_{|x-t|\geq\delta} \log |x-t| \, d\mu_{n,k}(t) - \int_{|x-t|\geq\delta} \log |x-t| \, d\mu_{w}(t) \\
	&+ \int_{|x-t|<\delta} \log |x-t| \, d\mu_{n,k}(t) - \int_{|x-t|<\delta} \log |x-t| \, d\mu_{w}(t).
\end{align*}

First, for $\delta < 1$, clearly 
\begin{equation}\label{Part1}
	\int_{|x-t|<\delta} \log |x-t| \, d\mu_{n,k}(t) = \sum_{\substack{j \neq k \\ |x-x_{n,j}| < \delta}} \log |x-x_{n,j}| \leq 0.
\end{equation}
To deal with the other pieces, first define the function $\chi_\delta(t;x)$ to be the indicator function for the set $ K_w \setminus B(x,\delta)$, where $B(x,\delta)$ is the ball of radius $\delta$ about $x$. We claim that for fixed $\delta>0$, the function
\begin{equation*}
	g(x) := \int_{B(x,\delta)} \log | x-t |\, d\mu_w(t),
\end{equation*}
is continuous. To see this, let $f_\delta (t ; x) := \chi_\delta(t;x) \log |x-t|$. Then,
\begin{equation*}
	g(x) := \int_{B(x,\delta)} \log | x-t |\, d\mu_w(t) = U^{\mu_w}(x) - \int_{K_w} f_\delta (t ; x) d\mu_w(t).
\end{equation*}
The first function on the right-hand side is continuous by Theorem~\ref{thm:potential}. To see that the latter is continuous, let $\{y_n\}_{n=1}^\infty \subset  K_w$ be a sequence converging to $x$. Then as $y_n\rightarrow x$, $f_\delta (t ; y_n)$ converges to $f_\delta (t ; x)$, and $|f_\delta (t ; y_n)| \leq \max\{ \log( \textrm{diam } K_w ), \log \frac1\delta \}$. Hence, by the bounded convergence theorem, $g(y_n) \rightarrow g(x)$.

Since the support of the measure $\mu_{w}$ is compact, we know the function $\log |x-t|$ is uniformly bounded above for $x,t \in K_w$. As $\delta\rightarrow0$,  $f_\delta (t ;x)$ is a decreasing sequence of integrable functions, which converge pointwise almost everywhere to $\log|x-t|$. Hence by the monotone convergence theorem, 
\begin{equation*}
	\int_{K_w} f_\delta (t ; x) \, d\mu_w(t) \rightarrow \int_{K_w} \log |x-t|  \, d\mu_w(t), \quad \delta\rightarrow0.
\end{equation*}
Hence, for any $x$, there exists a $1>\delta_x>0$ such that
\begin{equation}\label{dcont}
	- \int_{|x-t| < \delta_x} \log|x-t| \, d\mu_w(t) =  \int_{K_w} f_{\delta_x} (t;x ) \, d\mu_w(t)  - \int_{K_w} \log |x-t|  \, d\mu_w(t)  \leq \varepsilon/4.
\end{equation}
Furthermore, by the continuity argument in the previous paragraph, we can choose an $r_x <\delta_x$ so that for any $y\in K_w$ with $|y-x|<r_x$,
\begin{equation}\label{xcont}
	\left| \int_{|y-t| < \delta_x} \log|y-t| \, d\mu_w(t) -  \int_{|x-t| < \delta_x} \log|x-t| \, d\mu_w(t)  \right|  \leq \varepsilon/4.
\end{equation}
Again by compactness, we can cover $ K_w$ by some finite set $\{ B(y_i, r_{y_i}) \}_{i=1}^M$.  Moreover, there exists a $\delta>0$ such that for any $x\in K_w$, $B(x,\delta)\subset B(y_i, r_{y_i})$ for some $i=1,\ldots,M$. This will be the chosen $\delta$. Indeed, from~\eqref{dcont} and~\eqref{xcont}, and by $\delta<r_{y_i}<\delta_{y_i}$, 
\begin{align}
	- \int_{|x-t| < \delta} \log|x-t| \, d\mu_w(t) & \leq - \int_{|x-t| < \delta_{y_i}} \log|x-t| \, d\mu_w(t) \notag\\
	&\leq - \int_{|y_i-t| < \delta_{y_i}} \log|y_i-t| \, d\mu_w(t) + \varepsilon/4\leq \varepsilon/2. \label{Part2} 
\end{align}

Finally, we deal with the remaining integrals in~\eqref{WTS}. For any $x$, $\log|x-t|$ is continuous in the set $|x-t| > \delta$. 
The fact that $\mu_{n,k} \rightarrow \mu_w$ weak$^*$  implies by definition that there exists an $N_1\in\mathbb{N}$, such that
\begin{equation*}
	\int_{|x-t|\geq\delta} \log|x- t|\, d\mu_{n,k}(t) \leq \int_{|x-t|\geq\delta} \log|x - t|\,d\mu_w(t) + \varepsilon/4,\;\; \text{ if }\;\; n>N_1.
\end{equation*}
Moreover, for any non-negative integers $k_1 \neq k_2$, we find for some $C>0$,
\begin{align*}
	& \Big| \int_{|x-t|\geq\delta} \log|x- t|\, d\mu_{n,k_1}(t) - \int_{|x-t|\geq\delta} \log|x- t|\, d\mu_{n,k_2}(t)   \Big| \\ = &\Bigg| \frac1{\nn} \sum_{{j \neq k_1\atop |x-x_{n,j}| \geq \delta}} \log |x-x_{n,j}| - \frac1{\nn} \sum_{{j \neq k_2 \atop |x-x_{n,j}| \geq \delta}} \log |x-x_{n,j}| \Bigg| 
		 \leq  \frac1{\nn} \left| \log \left( \frac{\textrm{diam}(K_w)}{\delta} \right) \right|.
\end{align*}
The right-hand side is small as $n\rightarrow\infty$, so we can choose $N_2>N_1$ such that for $n>N_2$, and $0\leq k_1,k_2 \leq n$, 
$$\left| \int_{|x-t|\geq\delta} \log|x- t|\, d\mu_{n,k_1}(t) - \int_{|x-t|\geq\delta} \log|x- t|\, d\mu_{n,k_2}(t)   \right|<\varepsilon/4.$$ 
This implies that for $\nn>N_2$ and $0\leq\kk\leq\nn$, 
\begin{equation} \label{Part3}
	\int_{|x-t|\geq\delta} \log|x- t|\, d\mu_{n,k}(t) \leq \int_{|x-t|\geq\delta} \log|x - t|\,d\mu_w(t) + \varepsilon/2.
\end{equation}
Furthermore, by standard arguments from the compactness of the set $ K_w$, we can also choose $N > \max\{N_1,N_2\}$ to be independent of $x$. See [Taylor, Lemma 2.4.12].

Combining~\eqref{Part1},~\eqref{Part2}, and~\eqref{Part3} with~\eqref{WTS} yields the desired result.
\end{proof}

\subsection{Proof of Lemma~\ref{lem:P1}}\label{ssec:P1}
\begin{proof}
	Let $\varepsilon>0$ be given, and $K_w = \textrm{Supp}(\mu_w)$ as above. To prove the lemma, it will be enough to show that
	\begin{equation*}
		\Bigg| \log\Bigg( w(x_{\nn,\kk})^n \prod_{|x_{\nn,\kk} - x_{\nn,\jj}| \geq \delta} |x_{\nn,\kk} - x_{\nn,\jj}|  \Bigg)^{1/n} - (-F_w)\Bigg| < \varepsilon.
	\end{equation*}
	First, notice that
	\begin{equation*}
		\log\Bigg(  \prod_{|x_{\nn,\kk} - x_{\nn,\jj}| \geq \delta} |x_{\nn,\kk} - x_{\nn,\jj}| \Bigg)^{1/n} =  \int_{|t - x_{\nn,\kk}|\geq\delta} \log\left| t - x_{\nn,\kk} \right|\, d\mu_{\nn,\kk}(t),
	\end{equation*}
	and of course
	\begin{equation*}
		\log\left( w(x_{\nn,\kk})^n \right)^{1/n} = - Q(x_{\nn,\kk}).
	\end{equation*}
	Furthermore, we have already seen from~\eqref{UQF} that
	\begin{equation}
		U^{\mu_w}(x) +Q(x) = F_w, \quad \forall x\in K_w \subset\mathbb{R}.
	\end{equation} 
	Thus, we estimate
	\begin{align*}
		& \Bigg| \log\Bigg( w(x_{\nn,\kk})^n \prod_{|x_{\nn,\kk} - x_{\nn,\jj}| \geq \delta} |x_{\nn,\kk} - x_{\nn,\jj}| \Bigg)^{1/n} - (-F_w) \Bigg|  \\
			= & \left| \log\Bigg( w(x_{\nn,\kk})^n \prod_{|x_{\nn,\kk} - x_{\nn,\jj}| \geq \delta} |x_{\nn,\kk} - x_{\nn,\jj}| \Bigg)^{1/n} + U^{\mu_w}(x_{\nn,\kk}) + Q(x_{\nn,\kk}) \right|  \\
			\leq &\underbrace{\left| \int_{|t - x_{\nn,\kk}|\geq\delta} \log\left| t - x_{\nn,\kk} \right|\, d\mu_{\nn,\kk}(t) - \int_{|t - x_{\nn,\kk}|\geq\delta} \log\left| t - x_{\nn,\kk} \right|\, d\mu_w(t) \right|}_{A}\\
			&\qquad \qquad + \underbrace{\left| \int_{|t - x_{\nn,\kk}|<\delta} \log\left| t - x_{\nn,\kk} \right|\, d\mu_w(t) \right|}_{B} + \left| -Q(x_{\nn,\kk}) + Q(x_{\nn,\kk}) \right|.
	\end{align*}
	The last term is equal to zero, so it is left to show that there exists a $\delta>0$ and $N\in\mathbb{N}$ independent of $\nn$ and $\kk$ such that $A < \varepsilon/2$ and $B < \varepsilon/2$. The proof for the quantity A is shown in the proof of Theorem~\ref{lem:numer}, and the proof for B follows essentially from the proof of~\cite[Theorem 2.4.6]{taylor2008}, so we forgo the details here.
\end{proof}

\end{document}